\documentclass[12pt]{amsart}

\RequirePackage[nonfrench]{kotex}
\usepackage{kotex}

\usepackage{epsfig}
\usepackage{amsfonts}
\usepackage{amssymb}
\usepackage{amsmath}
\usepackage{amsthm}
\usepackage{amscd}
\usepackage{amstext}
\usepackage{latexsym}

\usepackage{enumerate}

\newcommand{\CC}{\mathbf{C}}
\newcommand{\QQ}{\mathbf{Q}}

\newcommand{\JJ}{\mathcal{J}}

\newcommand{\OO}{\mathcal{O}}

\newcommand{\de}{\partial}

\newcommand{\mfa}{\mathfrak{a}}
\newcommand{\mfb}{\mathfrak{b}}

\newcommand{\mfm}{\mathfrak{m}}

\newcommand{\ep}{\epsilon}

\newcommand{\qa}{\quad}
\newcommand{\vp}{\varphi}

\newcommand{\Ld}{\Lambda}

\newcommand{\noi}{\noindent}

\providecommand{\abs}[1]{\left|#1\right|}

\theoremstyle{plain}

\newtheorem{theorem}{Theorem}[section]

\newtheorem{thm}[theorem]{Theorem}

\newtheorem{corollary}[theorem]{Corollary}
\newtheorem{proposition}[theorem]{Proposition}
\newtheorem{definition}[theorem]{Definition}

\newtheorem{question}[theorem]{Question}

 \newtheorem{example}[theorem]{\textnormal{\textbf{Example}}}

\theoremstyle{remark}

\DeclareMathOperator{\lct}{lct}

\begin{document}

\title{On a question of Teissier} 


\subjclass[2010]{}

\author{Dano Kim}

\date{}

\maketitle

\begin{abstract}  

Using a  result of Demailly and Pham on log canonical thresholds, we give  an upper bound for polar invariants from a question of Teissier on hypersurface singularities. This provides a weaker alternative upper bound compared to the one conjectured by Teissier. 

\end{abstract}

\section{Introduction} 

 Let $f(z_1, \ldots, z_n)$ be a holomorphic function defined near $0 \in \CC^n$ such that the hypersurface $f= 0$ has an isolated singularity at $0$. In the context of a series of extensive research on such singularities (cf. \cite{T73}, \cite{T77}, \cite{T80}), Teissier considered the \emph{polar invariant} $\theta(f)$ which measures the rate of vanishing of the Jacobian ideal $J_f$ of $f$ with respect to that of the maximal ideal $\mfm$ of $0 \in \CC^n$. Moreover, he considered the (well-defined) polar invariants $\theta(f_1), \ldots, \theta(f_{n-1})$ where $f_j$ denotes the restriction of $f$ to general $j$-codimensional planes containing $0 \in \CC^n$. In this paper, we give the following upper bound for the particular combination of these polar invariants from one of the questions in \cite{T80}  (cf. \cite{T12}, \cite{Li83}, \cite{L84}).

 \begin{thm}\label{main}
 
  Let $\lct(\mfm \cdot J_f)$ be the log canonical threshold at $0 \in \CC^n$ of the product ideal $\mfm \cdot J_f$. We have 
 
 \begin{equation}\label{main1}
 \frac{1}{1+ \theta (f)} + \frac{1}{1+ \theta (f_1)} + \ldots + \frac{1}{1+ \theta (f_{n-1})} \le  \lct (\mfm \cdot J_f).   
 \end{equation}
 
 \end{thm}

In fact, in his question~\cite[p.7]{T80}, Teissier conjectured that one can put the \emph{Arnold exponent} $\sigma(f)$ of $f$ at $0$, in the place of $ \lct (\mfm \cdot J_f)$ in \eqref{main1}. The Arnold exponent $\sigma(f)$ is an invariant which is defined from certain natural asymptotic expansions of integrals (cf. \cite{M74}, \cite{AGV84}).     It is related to log canonical thresholds by $\lct (f) = \min \{ \sigma(f), 1 \}$, cf. \cite[\S 9]{Ko97}.  It is also known as (i.e. equal to) the \emph{minimal exponent} which is defined in terms of Bernstein-Sato polynomials, the equality being due to \cite{M75}, \cite{M74}.

  Since $f$ belongs to the integral closure of $\mfm \cdot J_f$ in the local ring at $0$ by \cite{BS74} (cf. \cite[Cor. 11.19]{D11}), we have 
\begin{equation}\label{lctm}
\lct (\mfm \cdot J_f)  \ge \lct (f) 
\end{equation} (note that strict inequality holds here for many examples, e.g. $f(x,y) = y^2 + x^3$). In view of these, Theorem~\ref{main} provides an upper bound, weaker than $\sigma(f)$, for the LHS of \eqref{main1} when the hypersurface $(f=0)$ fails to have rational singularities (which is equivalent to $\sigma(f) \ge 1$ by \cite{S93}). When $(f=0)$ does have rational singularities, it will be interesting to further compare $\lct (\mfm \cdot J_f)$ and $\sigma(f)$. We also remark that, at least in a certain sense, $\lct (\mfm \cdot J_f)$ can be regarded as `not too distant' from the singularity of $f$ itself, in view of  \cite[Prop. 3.8]{ELSV} which says that $J_f$ is included in the multiplier ideal $\JJ((f)^{1-\ep})$ for all $\ep > 0$.

 On the other hand, very recently the conjectured bound was shown by B. Dirks and M. Musta\c{t}\v{a} \cite{DM20}, building on the approaches of \cite{EM20}, \cite{L84} which, in particular, reduces the problem to the `one codimension' version. \footnote{This paper was written after we got interested from \cite{EM20} and  before the appearance of \cite{DM20}.}  The method of \cite{DM20} uses, among other things, Saito's theory of mixed Hodge modules and the theory of Hodge ideals (cf. \cite{MP19}, \cite{MP20}). In contrast, our method for the weaker result is completely different and does not use these theories. \footnote{But our method does not deal with the Arnold exponent directly. It will be interesting if the method of this paper can be further combined with the complex analytic aspect of the Arnold exponent = the minimal exponent.}    It uses singularities of plurisubharmonic  functions and  especially the following consequence of a fundamental result of Demailly and Pham~\cite{DH}.

\begin{theorem}

 Let $\vp$ be a psh function defined near $0 \in \CC^n$. Suppose that there exists $C \ge 0$ such that   $ \vp \ge C \log \abs{z} + O(1)$ as in \eqref{bb}. 
 Let $\lct$ be the log-canonical threshold at $0$ and $e_k$ be the $k$-th Lelong number at $0$. Let $L(\vp)$ be the  Łojasiewicz exponent of $\vp$ at $0$ (see Definition~\ref{loja}). For each $j= 1, \ldots, n-1$, let $\Ld_{j} = H_1 \cap \ldots \cap H_{j}$ where $H_1, \ldots, H_j$ are general hyperplanes containing $0 \in \CC^n$. 
Then we have

\begin{align*}
 \lct (\vp) &\ge  \frac{e_{n-1} (\vp)}{e_n (\vp)} +  \frac{e_{n-2} (\vp)}{e_{n-1} (\vp)} + \ldots +  \frac{1}{e_1 (\vp)} \\
  &\ge  \frac{1}{ L( \vp ) } + \frac{1}{ L( \vp|_{\Ld_1} ) }  + \ldots +  \frac{1}{ L( \vp|_{\Ld_{n-1}} ) }.
  \end{align*}

\end{theorem}

 \noi Here the first inequality is from  \cite{DH} and the second due to Proposition~\ref{eel}. Our arguments for Theorem~\ref{main} use only the `algebraic' case of  plurisubharmonic functions, i.e. those arising from ideals : thus it may be possible that they can be translated in the algebraic language (cf. \cite{BF16}, \cite{B17} for some closely related developments in this regard). We remark that the above main result of \cite{DH} is proved for general psh functions (with isolated singularities) by reducing to the algebraic case (which was further elaborated  in \cite{B17}).   
 It is also interesting that our method for Theorem~\ref{main} does not use the `one codimension' version of the main result of \cite{DH}.  In fact, such a statement is an open question :

\begin{question}[Hoang Hiep Pham~\cite{H19}]\label{Pham}
Let $\vp$ be a psh function with isolated singularities at $0 \in \CC^n$. Do we have 

$$ \lct(\vp) \ge \lct_{n-1} (\vp) + \frac{e_{n-1} (\vp)}{e_n (\vp)} $$

\noi where $\lct_{n-1} (\vp)$ denotes the supremum of $\lct (\vp|_H)$ when $H$ ranges over hyperplanes in $\CC^n$ containing $0 \in \CC^n$ ? 

\end{question}

 \noi We  hope that the interface of analytic and algebraic ideas as in this paper  will help shed light on Question~\ref{Pham}.  
 \\

\noi \textbf{Acknowledgements.}
The author is grateful to Mircea  Musta\c{t}\v{a} and Alexander Rashkovskii for  helpful comments. 
This research was supported by Basic Science Research Program through NRF Korea funded by the Ministry of Education (2018R1D1A1B07049683).
\\

\section{Some preliminaries and the proof} 

\subsection{Plurisubharmonic singularities} 

 We refer to \cite{D11} (and its predecessors) for the definition and basic properties of plurisubharmonic (i.e. psh) functions. 
 Let $u$ be a psh function defined near $0 \in \CC^n$. We will say that $u$ is singular at $0$ if $u(0) = -\infty$. 

Following \cite{D11},  we will say that two psh functions $u$ and $v$ have equivalent singularities if their difference is locally bounded, i.e. if $u = v + O(1)$ where $O(1)$ refers to a function which is locally bounded near every point.   Also we will say that $u$ is less singular (resp. more singular) than $v$ when $u \ge v + O(1)$ (resp. $u \le v + O(1)$).

 In this paper, we will be mostly concerned with psh singularities arising from ideals. When $\mfa$ is an ideal (or an ideal sheaf) locally generated by holomorphic functions $f_1, \ldots, f_m$, we will use the notation 
 
 \begin{equation}\label{ideal}
 u = c \log \abs{\mfa} 
 \end{equation}
 \noi to refer to the psh function (or more precisely the equivalence class of psh functions under the above relation $u = v + O(1)$) defined by $u = c \log (\sum^m_{j =1} \abs{f_j}^{}) $.  It is well known from \cite{D11} that the singularity equivalence class of $u$ is well-defined by the ideal (or the ideal sheaf), independent of choices of generators $f_1, \ldots, f_m$ (see also \cite[Prop. 3.1]{K14} for some basic exposition). The following is convenient when dealing with integral closure of ideals. 

\begin{proposition}\label{closure}

 Suppose that $\mfa$ and $\mfb$ are two coherent ideal sheaves whose integral closures coincide. Then we have $\log \abs{\mfa} = \log \abs{\mfb} + O(1)$ for the psh functions defined above as in \eqref{ideal} (for $c=1$).

\end{proposition}

\begin{proof} 

 This is well-known, we just recall it explicitly from the literature. It suffices to consider the case when $\mfb = \overline{\mfa}$. Since $\mfa \subset \overline{\mfa}$,  we first have $\log \abs{\mfa} \le \log \abs{\overline{\mfa}} + O(1)$. 
 
 On the other hand, suppose that $\mfa$ is generated by $h_1, \ldots, h_m$. From \cite[p.799, Thm. 2.1]{LT08} (cf. \cite[(9.6.10)]{L}), we have $\log \abs{h_j} \le \log \abs{\mfa}$. Thus we have  $\log \abs{\overline{\mfa}} = \log \sum \abs{h_j} + O(1) = \log \max \abs{h_j} + O(1) \le \log \abs{\mfa} + O(1) $.  
\end{proof}

 For psh functions, Lelong numbers are  important and fundamental measure of their singularities  (cf. \cite{D11}). When a psh function $u$ has isolated singularities (i.e. locally bounded outside the point $0 \in \CC^n$), not only the first (usual) Lelong number $e_1 (u)$, but higher Lelong numbers $e_2 (u), \ldots, e_n(u)$ are defined thanks to work of Demailly (cf. \cite{D93}), which can be expressed as follows: 

\begin{equation}\label{kth}
  e_k (u) := e_k (u, 0) = \int_{\{0\}} (dd^c u)^k \wedge (dd^c \log \abs{z} )^{n-k} 
\end{equation}

\noi for $k= 1, \ldots, n$. When $u$ arises from an ideal $\mfa$ (which then should be zero-dimensional)  in the sense of \eqref{ideal}, the higher Lelong numbers correspond exactly to mixed multiplicities defined from commutative algebra (cf. \cite{T73}, \cite{L}, \cite{P15}, \cite{B17}). In fact, this holds in the generality of mixed Monge-Ampère masses which include higher Lelong numbers as special cases. For psh functions $u_1, \ldots, u_n$ having isolated singularities at $0 \in \CC^n$, we define and denote their (residual) mixed Monge-Ampère mass at $0 \in \CC^n$ by 

$$ e(u_1, \ldots, u_n) = \int_{\{0\}}  (dd^c u_1) \wedge \ldots \wedge (dd^c u_n) .$$

\noi We have the following relation between mixed Monge-Ampère mass and mixed multiplicities. 

\begin{proposition}

 Let $\mfa_1, \ldots, \mfa_n \subset \OO_{0, \CC^n}$ be zero-dimensional ideals at $0 \in \CC^n$. For the psh functions $u_k = \log \abs{\mfa_k}, k = 1, \ldots, n$, we have  the equality  $$ e(u_1, \ldots, u_n) = \mu (\mfa_1, \ldots, \mfa_n) $$ where the RHS is the mixed multiplicity of $\mfa_1, \ldots, \mfa_n$. 

\end{proposition}

\begin{proof} 

The case when $\mfa_1 = \ldots = \mfa_n$ is due to  the fundamental result \cite[Lem. 2.1]{D09}. The general case follows from  the polarization (cf. \cite{R11}, \cite{KR18}) as was noted in \cite[Cor. 4.2]{R11}. 

\end{proof} 

\subsection{Łojasiewicz exponents and polar invariants}

Now we define the Łojasiewicz exponent of a psh function. 

\begin{definition}\label{loja}
 
  Let $\vp$ be a psh function germ at $0 \in \CC^n$ with isolated singularities, i.e. $\vp$ is locally bounded outside $0$. Suppose that there exists $C \ge 0$ such that 
 
 \begin{equation}\label{bb} 
  \vp \ge C \log \abs{z} + O(1)
  \end{equation} 
where  $\abs{z}^2 = \abs{z_1}^2 + \ldots + \abs{z_n}^2$. 
   We define the Łojasiewicz exponent of $\vp$ to be the infimum of such $C$'s and denote it by $L(\vp)$.

\end{definition}

 This agrees with the more usual Łojasiewicz exponents when $\vp$ arises from an ideal (cf. \cite[(1.7)]{T77}, \cite{BF16}). An example of a psh function $\vp$ with isolated singularities but not satisfying the condition in this definition can be found e.g. in \cite[Ex. 2.7]{KR20} (cf. \cite[p.354]{R10}). Note that if $\vp$ satisfies \eqref{bb}, then so does its restriction to $\Ld_j$ and thus we can define $L(\vp|_{\Ld_j})$.

 \begin{proposition}\label{eel}
 Let $\vp$ be as in Definition~\ref{loja}.

\begin{enumerate}
\item
 We have $\displaystyle \frac{e_{n-1} (\vp)}{e_n (\vp)} \ge \frac{1}{L (\vp)} $ where $L(\vp)$ is the Łojasiewicz exponent of $\vp$.

\item 
 For each $j= 1, \ldots, n-1$, let $\Ld_{j} = H_1 \cap \ldots \cap H_{j}$ where $H_1, \ldots, H_j$ are general hyperplanes containing $0 \in \CC^n$.  Then we have 
 $$ \frac{e_{n-j-1} (\vp) }{e_{n-j} (\vp)} \ge \frac{1}{ L( \vp|_{\Ld_j} ) } . $$
 
 \end{enumerate}
 
\end{proposition}

\noi In (2), `general' is in the sense of `almost all' as in \cite{S74} which is used in the following proof. 

\begin{proof} 

 Let $u := \log \abs{\mfm} = \log \abs{z}$. 
 In terms of mixed Monge-Ampère masses, we have $e_{n-1} (\vp) = e (\vp, \ldots, \vp, u)$ and $e_n (\vp) = e(\vp, \ldots, \vp, \vp)$. Whenever $\vp \ge C u + O(1)$, we have $  e (\vp, \ldots, \vp, u) \ge \frac{1}{C} e(\vp, \ldots, \vp, \vp)$ by Demailly's comparison theorem~\cite{D93}. Since $L(\vp)$ is the infimum of such $C$'s, (1) follows.   
 
 Now when we restrict (1) to $\Ld_j$, we obtain (2) from \cite{S74}. More precisely, let $T$ and $S$ be the closed positive currents of bidegree $(n-j, n-j)$ given by  
 
 $$ T = (dd^c \vp) \wedge \ldots \wedge (dd^c \vp) \wedge (dd^c u), \; \; S = (dd^c \vp) \wedge \ldots \wedge (dd^c \vp) \wedge (dd^c \vp)   .$$

 From the fundamental slicing theory of closed positive currents of \cite[\S 11, p.136]{S74}, the Lelong number of $T$ (resp. of $S$) at $0$ is equal to the Lelong number of its slice by almost every $\Ld_j$ (i.e. restriction to $\Ld_j$), a plane of codimension $j$. Here the Lelong number of $T$ (resp. $S$) is nothing but $e(\vp, \ldots, \vp, u, u, \ldots, u)$
 (resp. $e(\vp, \ldots, \vp, \vp, u, \ldots, u)$) where $u$ is repeated $j+1$ times (resp. $j$ times).  Hence they are equal to 
 
 $$ e(\vp|_{\Ld_j}, \ldots, \vp|_{\Ld_j}, u|_{\Ld_j} )   \quad \text{          resp. }    e(\vp|_{\Ld_j}, \ldots, \vp|_{\Ld_j}, \vp|_{\Ld_j} )  $$
 
\noi which are mixed Monge-Ampère masses taken on $\Ld_j$ (thus taking $n-j$ arguments, of course). 
Therefore (1) applied on $\Ld_j$ yields (2). 
 
\end{proof} 

\begin{corollary}\label{polar}

 Let $f$ be as in the beginning of the introduction. For Teissier's polar invariants, we have the following relation when we take $\vp = \log \abs{\mfm \cdot J_f}$, for every $j = 0, 1, \ldots, n-1$,

 $$ \frac{e_{n-j-1} (\vp) }{e_{n-j} (\vp)} \ge \frac{1}{1+ \theta (f_{j}) }  $$
 
\noi where $ f_0 := f$ and for $j \ge 1$,  $f_j := f|_{\Ld_j}$ and $\Ld_{j} = H_1 \cap \ldots \cap H_{j}$ for general hyperplanes $H_1, \ldots, H_j$ containing $0 \in \CC^n$.  
 
\end{corollary} 

\begin{proof} 

 Let us first recall from  \cite[p.270, (1.7) Corollaire 2]{T77}  that, in general, the polar invariant $ \theta (f)$ is equal to the Łojasiewicz exponent of the Jacobian ideal $J_f$, which is equal to $L(\log \abs{J_f})$ in our notation. \footnote{The same argument can be also checked using \cite[Thm. 10.10]{B20} and the definition of $\theta(f)$.}  We will apply this to $f_j$, for which we also need the fundamental fact that the Jacobian ideal of the restriction  $f_j$ and the restriction of the Jacobian ideal $J_f$  coincide up to integral closure, taken from \cite{T73} (cf. \cite[p.21]{P15}). In our notation of psh functions, on $\Ld_j$, we have (using Proposition~\ref{closure})
 
 \begin{equation}\label{jfj}
 \log \abs{J_{f_j}} = \psi|_{\Ld_j} + O(1) 
 \end{equation} where $\psi = \log \abs{J_f}$. 
  Now since $\vp =  \log \abs{\mfm J_f} = \psi + \log \abs{\mfm} + O(1)$, we have $L (\vp |_{\Ld_j} ) = L (\psi |_{\Ld_j})  + L (\log \abs{\mfm} |_{\Ld_j}) = L (\psi |_{\Ld_j})+ 1$ which is equal to $L(\log \abs{J_{f_j}}) + 1$ by \eqref{jfj}.  We then have $L (\vp |_{\Ld_j} ) = 1+ \theta (f_j)$ by the first sentence above : hence  the assertion follows from Proposition~\ref{eel}, (2). 
\end{proof}

  Now we complete the proof of Theorem~\ref{main}. 

\begin{proof}[Proof of Theorem~\ref{main}]

 Let $J_f  = ( \frac{\de f}{\de z_1}, \ldots,   \frac{\de f}{\de z_n} )$ be the Jacobian ideal of $f$.  Briançon and Skoda~\cite{BS74} showed (cf. \cite{D11}) that $f$ belongs to the integral closure of the ideal $$(z_1  \frac{\de f}{\de z_1}, \ldots, z_n  \frac{\de f}{\de z_n} ) \subset (z_1, \ldots, z_n) \cdot J_f =: \mfa .$$

Letting $\vp = \log \abs{\mfa}$, we apply the main result of \cite{DH}, Theorem 1.2, to $\vp$ which has isolated singularities. Together with Corollary~\ref{polar}, it gives

\begin{equation}\label{lct_polar}
 \lct (\vp) \ge \frac{1}{e_1 (\vp)} + \frac{e_1(\vp)}{e_2(\vp)} + \ldots + \frac{e_{n-1} (\vp)}{e_n (\vp)} \ge \sum^{n-1}_{j=0} \frac{1}{1+ \theta (f_{j}) } .
\end{equation}

\noi Here we note that certainly $e_1 (\vp) >0 $ holds, which then implies  $e_j (\vp) >0$ by \cite[Cor. 2.2]{DH} for $j = 2, \ldots, n$.  
\end{proof}

\begin{example}

 Let $f(z_1, \ldots, z_n) = z_1^d + \ldots + z_n^d$. Then equality holds in all of \eqref{main1}, \eqref{lctm}, \eqref{lct_polar}. Note that $(z_1^d, \ldots, z_n^d) \subset  \mfm \cdot J_f  \subset \mfm^d$ and thus the integral closure of $ \mfm \cdot J_f $ equals $\mfm^d$. 

\end{example}

 We remark that further investigations on the equality cases of the above inequality of \cite{DH} (cf. \cite{R15}, \cite{B17}) and of the inequalities in Proposition~\ref{eel} would be of great interest.

\footnotesize

\bibliographystyle{amsplain}

\qa

\qa

\normalsize

\noi \textsc{Dano Kim}

\noi Department of Mathematical Sciences \& Research Institute of Mathematics

\noi Seoul National University, 08826  Seoul, Korea

\noi Email address: kimdano@snu.ac.kr

\end{document}